\documentclass{article}
\usepackage{amsmath}
\usepackage{amssymb}
\usepackage{amsthm}
\usepackage{enumerate}

\makeatletter
\def\iddots{\mathinner{\mkern1mu\raise\p@
    \hbox{.}\mkern2mu\raise4\p@\hbox{.}\mkern2mu
    \raise7\p@\vbox{\kern7\p@\hbox{.}}\mkern1mu}}
\def\adots{\mathinner{\mkern2mu\raise\p@\hbox{.} 
 \raise7\p@\vbox{\kern7\p@\hbox{.}}\mkern1mu}}
\makeatother

\newtheorem{theo}{Theorem}[section]\newtheorem{pr}[theo]{Proposition}\newtheorem{col}[theo]{Corollary}\theoremstyle{definition}\newtheorem{defn}[theo]{Definition}\newtheorem{remark}[theo]{Remark}\newtheorem{ex}[theo]{Example}

\def\rank{\mathop{\rm rank}\nolimits}\def\dim{\mathop{\rm dim}\nolimits}\def\diag{\mathop{\rm diag}\nolimits}\def\Ker{\mathop{\rm Ker}\nolimits}\def\ad{\mathop{\rm ad}\nolimits}\def\C{\mathop{\mathbb{C}}\nolimits}\def\Z{\mathop{\mathbb{Z}}\nolimits}\def\N{\mathop{\mathbb{N}}\nolimits}\def\g{\mathop{\mathfrak{g}}\nolimits}\def\h{\mathop{\mathfrak{h}}\nolimits}\def\l{\mathop{\mathfrak{l}}\nolimits}\def\L{\mathop{\mathfrak{L}}\nolimits}\def\Hom{\mathop{\rm Hom}\nolimits}\def\sl{\mathop{\mathfrak{sl}}\nolimits}\def\V{\mathop{\mathcal{V}}\nolimits}\def\1{\mathop{\overline {1}}\nolimits}\def\0{\mathop{\overline {0}}\nolimits}\def\Lie{\mathop{\rm Lie}\nolimits}

\begin{document}
%
\begin{center}
{\huge  
A class of Lie algebras who contains a class of Kac-Moody algebras
\footnote{
{\bf 2010 Mathematic Subjects Classification}: Primary 17Bxx 
\\
Keywords and phrases: Kac-Moody algebras, standard pentads, PC Lie algebras}
}
\end{center}
\vspace{50truept}

\begin{center}{Nagatoshi SASANO}\end{center}
\begin{abstract}
V.Kac showed that from a local Lie algebra, we can construct a graded Lie algebra whose local part is the given local Lie algebra.
In other words, using standard pentads, we can embed given Lie algebra and its representation into a larger graded Lie algebra.
As special cases of Lie algebras associated with standard pentads, we have the notion of PC Lie algebras.
Our aim of this paper is to show that the class of PC Lie algebras contains the class of Kac-Moody algebras, that is, to show that the notion of PC Lie algebras is an extension of Kac-Moody algebras.
\end{abstract}

\section {INTRODUCTION}
In this paper, we shall consider a certain class of Lie algebras, the class of ``PC Lie algebras''.
And, we shall compare it with the class of Kac-Moody algebras.
First, let us see a brief history around the theory of PC Lie algebras in this INTRODUCTION.
\subsection {A brief explanation on fundamental results (standard pentads)}
A PC Lie algebra is a Lie algebra which is constructed with a ``standard pentad'' (see Definition \ref {defn;stap} below).
To define the notion of PC Lie algebras, we need the notion of standard pentads.
\par
The motivation of the ``standard pentads theory'' comes from the theory of prehomogeneous vector spaces of parabolic type, which is studied by H.Rubenthaler.
Roughly, a prehomogeneous vector space of parabolic type is a special prehomogeneous vector space which are obtained from a local part of some graded finite-dimensional semisimple Lie algebra.
\par 
To consider ``converse problem'' of the preceding theory of prehomogeneous vector space of parabolic type, the author defined the notion of ``standard quadruplet'' in 2014.
Later, the author developed the notion of ``standard quadruplets'' to the notion of ``standard pentads'' in \cite {Sa3}.
\begin{defn}[\text {standard pentads, see \cite [Definitions 2.1 and 2.2]{Sa3}}]\label {defn;stap}
Let $\g $ be a Lie algebra with non-degenerate and invariant bilinear form $B$, $(\rho ,V)$ a representation\footnote {
In this paper, we consider a representation $\rho $ of $\g $ on $V$ as a linear map $\rho \colon \g\otimes V\rightarrow V$ satisfying $\rho ([a,b]\otimes v)=\rho (a\otimes \rho (b\otimes v))-\rho (b\otimes \rho (a\otimes v))$ for any $a,b\in \g $ and $v\in V$ (cf. \cite [Notation 1.1]{Sa3}).}
 of $\g $, $(\varrho ,\V)$ a $\g $-submodule of $\Hom (\g ,\C)$.
If a pentad $(\g,\rho,V,\V,B)$ satisfies the following conditions, we say that this pentad is {\it a standard pentad}:
\begin{enumerate}
\item {the restriction of the canonical paring $\langle \cdot ,\cdot \rangle \colon V\times \Hom (V,\C )$ to $V\times \V$ is non-degenerate},
\item {there exists a linear map $\Phi _{\rho }\colon V\otimes \V \rightarrow \g $, called the $\Phi $-map of the pentad, satisfying an equation 
$$
B(a,\Phi _{\rho }(v\otimes \phi ))=\langle \rho (a\otimes v),\phi \rangle =-\langle v,\varrho (a\otimes \phi )\rangle 
$$
for any $a\in \g$, $v\in V$, $\phi \in \V$.
}
\end{enumerate}
\end{defn}
From a standard pentad, we can construct a local Lie algebra $\V \oplus \g \oplus V$ canonically.
As well-known, for a given local Lie algebra, there exists a graded Lie algebra whose local part is the given local Lie algebra (proved by V.Kac in \cite [p.1276, Proposition 4]{Kac2}).
\begin{theo}[\text {summary of \cite [Theorem 2.15]{Sa3}, \cite [p.125]{Sa4}}]\label{theo;liepen}
Let $(\g,\rho ,V,\V,B)$ be a standard pentad.
Then there exists a graded Lie algebra $L(\g,\rho,V,\V,B)=\bigoplus _{n\in \Z}V_n$ satisfying the followings.
\begin{itemize}
\item {We have
$$
V_0\simeq \g 
$$ 
as Lie algebras,
$$
V_{-1}\simeq \V,\qquad V_1\simeq V
$$
as $V_0 \simeq \g$-modules,
and that the restricted bracket product 
$$
[\cdot,\cdot ]\mid _{V_{1}\times V_{-1}}\colon V\times \V \simeq V_1\times V_{-1}\rightarrow V_0\simeq \g
$$
is identified with the $\Phi $-map of $(\g,\rho,V,\V,B)$.
}
\item {The graded Lie algebra $L(\g,\rho,V,\V,B)$ is minimal in the sense of \cite [p.1276, Definition 6]{Kac2}.}
\end{itemize}
We call this graded Lie algebra $L(\g,\rho,V,\V,B)$ {\rm the Lie algebra associated with the standard pentad $(\g,\rho,V,\V,B)$}.
The graded Lie algebra $L(\g,\rho,V,\V,B)$ is determined from $(\g,\rho,V,\V,B)$ uniquely\footnote {In \cite {Sa3}, the author showed the existence of $L(\g,\rho,V,\V,B)$ without using the V.Kac's results. This is the theory of standard pentads. After this, H.Rubenthaler showed similar results using V.Kac's results and H.Rubenthaler's theory of fundamental triplets in \cite {Ru1} (refer \S \ref {aim1} below). The theories of standard pentads and fundamental triplets have the same goal in the sense of \cite [p.125]{Sa4}. Thus, we have the claim of Theorem \ref {theo;liepen}.}.
\end{theo}
Using Theorem \ref{theo;liepen}, we can ``embed'' any finite-dimensional Lie algebra and its representation (independent to its prehomogeneity) into some graded Lie algebra in the same sense as prehomogeneous vector spaces of parabolic type.
In 2015, it is proved by the author that the prehomogeneity of a representation $(\rho, V)$ of the reductive group $G$ can be described by the algebraic structure of Lie algebras of the form $L(\Lie (G),d\rho,V,\Hom (V,\C), B)$ (for detail, see \cite {Sa2}).
Thus, using the notion of standard pentads, we can extend the theory of prehomogeneous vector spaces of parabolic type.
For example, in the theory of prehomogeneous vector spaces of parabolic type, the regularity of prehomogeneous vector spaces is expressed by $\sl _2$-subalgebras of graded Lie algebras.
Using the notion of standard pentads, the regularity of prehomogeneous vector spaces (not necessary of parabolic type) with $1$-dimensional scalar multiplication can be expressed by $\sl _2$-subalgebras of graded Lie algebras (see \cite [Theorem 2.6]{Sa4}).

\subsection {A brief explanation on fundamental results (PC Lie algebras)}
Here, we have a problem that ``how can we write such graded Lie algebras $L(\g,\rho,V,\V,B)$ using well-known Lie algebras?''.
To answer this problem, we have the notion of PC Lie algebras, which will be mainly discussed in this paper.
The definition of PC Lie algebras is as the following.
\begin{defn}[\text {PC Lie algebras, \cite [Definitions 2.1--2.4]{Sa5}}]
\label {defn;1}
Let $r,n$ be positive integers, $A\in M(r,r)$ an invertible matrix, $D=(d_{ij})\in M(r,n)$ a matrix, $\Gamma =\diag (\gamma _1,\ldots ,\gamma _n)\in M(n,n)$ an invertible diagonal matrix.
We define some objects from the data $(r,n;A,D,\Gamma )$ as follows:
\begin{itemize}
\item {
we let $\h ^r=\{(a_1,\ldots ,a_r)\mid a_1,\ldots ,a_r\in \C \}$ be an $r$-dimensional commutative Lie algebra spanned by $\epsilon _i=(\delta _{i1},\ldots ,\delta _{in})$ $(i=1,\ldots ,n)$ where $\delta _{ik}$ stands for the Kronecker delta,
}
\item {
we let $\C _D=\C e_1\oplus \cdots \oplus \C e_n$, $\C_{-D}=\C f_1\oplus \cdots \oplus \C f_n$ be $n$-dimensional vector spaces spanned by symbols $e_i$\rq{s} and $f_i$\rq{s} respectively,}
\item {
we define $\h $-module structures $\Box _D$, $\Box _{-D}$ on $\C _D$, $\C_{-D}$ by $\Box _D(\epsilon  _i\otimes e_j)=d_{ij}e_j$, $\Box _{-D}(\epsilon _i\otimes f_j)=-d_{ij}f_j$ for $1\leq i\leq r$, $1\leq j\leq n$ respectively,
}
\item {
we define a bilinear form $B_A$ on $\h ^r$ by
\begin{align*}
B_A
\left (
(a_1,\ldots ,a_n), (a_1^{\prime },\ldots, a_n^{\prime })
\right )
=\begin{pmatrix}a_1&\cdots &a_n\end{pmatrix}{}^t A^{-1}\begin{pmatrix}a_1^{\prime }\\ \vdots \\ a_n^{\prime }\end{pmatrix},
\end{align*}}
\item {
we define a pairing $\langle \cdot ,\cdot \rangle $ between $\C _D$ and $\C_{-D}$ by $\langle e_i,f_j\rangle =\delta _{ij}\gamma _i$.
}
\end{itemize}
Under these, we have a standard pentad $(\h ^r, \Box _D, \C _D, \C _{-D}, B_A)$ and denote it by $P(r,n,A,D,\Gamma )$.
We denote the corresponding Lie algebra to $P(r,n;A,D,\Gamma )$ by $L(r,n;A,D,\Gamma )$.
We call a Lie algebra of the form $L(r,n;A,D,\Gamma )$ {\it a PC Lie algebra}\footnote {the term {\it PC} comes from {\it P}entad of {\it C}artan type.}.
\end{defn}
In the theory of standard pentads, we have an isomorphism among a repetition of construction of graded Lie algebras, ``chain rule'' (see \cite [Theorem 3.26]{Sa3}).
This theorem gives a proof of the following theorems.
\begin{theo}[\text {\cite[Theorem 2.1]{Sa6}}]\label{theo;pcchain}
Let $\g$ be a finite-dimensional reductive Lie algebra, $(\rho ,V)$ a finite-dimensional completely reducible representation, $B$ a non-degenerate symmetric invariant bilinear form on $\g$ all defined over $\C$.
Then, a pentad $(\g,\rho,V,\Hom (V,\C),B)$ is standard and the corresponding Lie algebra $L(\g,\rho,V,\Hom (V,\C),B)$ is isomorphic to some PC Lie algebra.
\end{theo}
Using \cite [Theorem 3.28]{Sa5}, we can compute the datum $(r,n;A,D,\Gamma)$ of PC Lie algebra corresponding to $(\g,\rho,V,\Hom (V,\C),B)$ in Theorem \ref{theo;pcchain} explicitly.
Thus, to compute the structure of $L(\g,\rho,V,\V,B)$ under the conditions of Theorem \ref {theo;pcchain}, it is enough to compute the structure of PC Lie algebras.
The structure of each PC Lie algebra can be expressed with (reduced) contragredient Lie algebras, defined by V.Kac (see \cite {Kac2}).
\begin{theo}[\text {\cite [Theorem 3.2]{Sa6}}]\label{theo;pcstructure}
Let $(r,n;A,D,\Gamma )$ be an arbitrary pentad of Cartan type.
Then, we have an isomorphism
\begin{align*}
L(r,n;A,D,\Gamma )\simeq (Z\oplus G^{\prime }(C(A,D,\Gamma )))\oplus \Delta 
\end{align*}
where 
\begin{itemize}
\item {the matrix $C(A,D,\Gamma )$ is $n\times n$ square matrix defined by $C(A,D,\Gamma )=(C_{ij})=\Gamma \cdot {}^t D\cdot A\cdot D$, which is called ``the Cartan matrix'' of the pentad $(r,n;A,D,\Gamma )$},
\item {$G^{\prime }(C(A,D,\Gamma ))$ is the reduced contragredient Lie algebra whose Cartan matrix $C(A,D,\Gamma )$ (the term ``Cartan matrix'' of this item is in the sense of contragredient Lie algebras (see \cite [p.1279]{Kac2}))},
\item {$[L(r,n;A,D,\Gamma ), L(r,n;A,D,\Gamma )]=Z\oplus G^{\prime }(C(A,D,\Gamma ))$},
\item {$Z \subset (\text {the center of }L(r,n;A,D,\Gamma ))$},
\item {$\Delta $ acts on each grading component diagonally},
\item {$\dim Z=\rank D-\rank C(A,D,\Gamma )$, $\dim \Delta =r-\rank D$}.
\end{itemize}
\end{theo}
That is, via the theory of standard pentads, any finite-dimensional reductive Lie algebra and its finite-dimensional completely reducible representation can be reduced to the structure theory of some PC Lie algebras, or the theory of (reduced) contragredient Lie algebras.
In particular, the study of prehomogeneous vector spaces with a reductive group and its completely reducible representation is also reduced to the theory PC Lie algebras.
\par
Using the notion of PC Lie algebras, we can distinguish a (reduced) contragredient Lie algebra with a certain Cartan matrix and a Kac-Moody algebra with the same Cartan matrix.
For example, if we take 
$$
C=\begin{pmatrix}2&-2\\-2&2\end{pmatrix}, 
$$
the corresponding reduced contragredient Lie algebra $G^{\prime }(C)$ and Kac-Moody algebra $\g(C)$ are respectively 
$$
G^{\prime }(C)\simeq (\text {the loop algebra $\C [t,t^{-1}]\otimes \sl _2$})\not \simeq \left (\C [t,t^{-1}]\otimes \sl _2\right )\oplus \C K\oplus \C d\simeq \g (C)
$$
(the notations $\C K$, $\C d$ are as \cite [Chapter 7]{Kac1}).
To obtain these algebras, we need ``the common data $C$'' and ``different ways (the construction of reduced contragredient Lie algebras and the one of Kac-Moody algebras)''.
On the other hand, using the notion of PC Lie algebras, we can obtain these algebras with ``different data'' and ``the common way'':
\begin{align*}
&
\C [t,t^{-1}]\otimes \sl _2\simeq G^{\prime }(C)\simeq L\left (1,2;\begin{pmatrix}1/8\end{pmatrix},\begin{pmatrix}2&-2\end{pmatrix},\begin{pmatrix}4&0\\0&4\end{pmatrix}\right )
&
\\
&
\left (\C [t,t^{-1}]\otimes \sl _2\right )\oplus \C K\oplus \C d\simeq \g (C)\simeq L\left (3,2;\begin{pmatrix}1/8&0&0\\0&0&1\\0&1&0\end{pmatrix},\begin{pmatrix}2&-2\\0&0\\0&1\end{pmatrix},\begin{pmatrix}4&0\\0&4\end{pmatrix}\right )
&
\end{align*}
 (see \cite [Examples 3.14, 3.15]{Sa5}).


\subsection {The first aim of this paper, comparing with Kac-Moody algebras}\label{aim1}
The structure of each PC Lie algebra is determined in Theorem \ref {theo;pcstructure}.
So, our next interest is the whole class of PC Lie algebras, ``how large is the class of PC Lie algebras?''. 
In this paper, we shall consider this problem\footnote {
As above, the class of prehomogeneous vector spaces with a reductive group and its completely reducible representation is reduced to the class of PC Lie algebras with a certain algebraic property.
Moreover, regularity of prehomogeneous vector spaces with 1-dimensional scalar multiplication is also expressed by a class of PC Lie algebras with a certain algebraic property.
So, to understand the class of PC Lie algebras helps us to understand prehomogeneous vectos spaces.
}.
For this, we shall compare the class of PC Lie algebras and the one of Kac-Moody algebras.
\par
By the way, in 2017, H.Rubenthaler announced a similar theory to the theory of standard pentads, called the theory of fundamental triplets (see \cite [Definition 3.1.1]{Ru1}).
And moreover, he proved that the class of his new Lie algebras contains the the class of symmetrizable Lie algebras, which is the subclass of Kac-Moody algebras consists of ones whose Cartan matrix is symmetric (see \cite [Example 3.3.4]{Ru1}).
\par
H.Rubenthaler's result can be applied to the author's theory of standard pentads (see \cite [Theorem 1.3]{Sa4}).
So, in other words, we can say that our aim is to extend his result to the whole class of Kac-Moody algebras.
In in Theorem \ref{maintheo1}, we shall show that the class of PC Lie algebras contains the one of Kac-Moody algebras.

\subsection {The second aim of this paper, a ``HUGE'' Lie algebra which contains $\sl _2$ and its all finite-dimensional representations}
The second aim of this paper, we give a ``HUGE'' Lie algebra $L\left (\sl _2, \mathrm {f.d.}\right )$ which is beyond the notion of PC Lie algebras.
Using the theory of standard pentads, we can embed a given finite-dimensional semisimple Lie algebra and a given representation into some (PC) Lie algebra.
So, we have a natural question that ``For a given finite-dimensional semisimple Lie algebra, can we obtain and compute a Lie algebra such that the given semieimple Lie algebra and its ALL finite-dimensional representation can be embedded into it?''.
In this paper, we consider this problem in the case where the (semi)simple Lie algebra is $\sl _2$.
The Lie algebra $L\left (\sl _2, \mathrm {f.d.}\right )$ is a Lie algebra thus constructed. 
Although this Lie algebra is NOT a PC Lie algebra, we can construct this by extending the idea of PC Lie algebras.

\section {PC Lie algebras and Kac-Moody algebras}
As we have seen in the previous section, PC Lie algebras are expressed by (reduced) contragredient Lie algebras.
In this section, we shall compare these classes.
On the notion and notations of Kac-Moody algebras, refer \cite [Chapters 1,2]{Kac1}.
\subsection {Main theorem}
In \cite [Example 3.3.4]{Ru1}, H.Rubenthaler showed that a symmetrizable Lie algebra $\g (C)$, where $C$ is a symmetrizable matrix but is not necessary a generalized Cartan matrix, is obtained as $\g_{min}(\Gamma (\g _0,B_0,(\rho ,V)))$ from some fundamental triplet
\footnote {A fundamental triplet is a triplet $(\g _0,B_0,(\rho,V))$ which consists of a quadratic Lie algebra $(\g _0,B_0)$ and finite-dimensional $\g _0$-module $(\rho ,V)$ (\cite [Definition 3.1.1]{Ru1}). 
From a fundamental triplet, we can construct graded Lie algebras $\g _{max}(\Gamma (\g _0,B_0,\rho))$, $\g _{min}(\Gamma (\g _0,B_0,\rho))$ (\cite [\S 3.3]{Ru1}).
The goals of his construction theory of minimal Lie algebra $\g _{min}(\Gamma (\g _0,B_0,\rho))$ and of our theory of standard pentads are coincides.
That is, if we consider a standard pentad $(\g _0,\rho ,V,\Hom (V,\C ), B_0)$, the corresponding graded Lie algebra $L(\g _0,\rho ,V,\Hom (V,\C ), B_0)$ is isomorphic to $\g _{min}(\Gamma (\g _0,B_0,\rho))$ (\cite [Theorem 1.3]{Sa4}).
}.
Let us extend his results in the following theorem.
\begin{theo}\label {maintheo1}
The followings hold:
\begin{enumerate}
\item {for any invertible matrix $C$, the contragredient Lie algebra $G(C)\simeq G^{\prime }(C)$ belongs to the class of PC Lie algebras},
\item {for any symmetrizable matrix $C$ which has a symmetrization with real number entries (see \cite [\S 2.1]{Kac1}), the reduced contragredient Lie algebra $G^{\prime }(C)$ belongs to the class of symmetric PC Lie algebras. A symmetric PC Lie algebra is a PC Lie algebra assaciated with a standard pentad having a symmetric bilinearr form (see \cite [Definition 1.14]{Sa5})},
\item {for any (not necessarily a generalized Cartan matrix) square matrix $C$, the Lie algebra $\g (C)$ (defined in \cite [\S 1.3]{Kac1}) belongs to the class of PC Lie algebras, in particular, any Kac-Moody algebra belongs to the class of PC Lie algebras},
\item {for any (not necessarily a generalized Cartan matrix) square matrix $C$, there exists a PC Lie algebra such that its derived algebra is isomorphic to the contragredient Lie algebra with Cartan matrix $C$}.
\end{enumerate}
\end{theo}
\begin{proof}
\begin{enumerate}
\item {See \cite [Theorem 3.11]{Sa5}.}
\item {
Let $n $ be the order of $C$, $l$ the rank of $C$, i.e. $C$ is an $n\times n$ square matrix with rank $l$.
From the assumption that $C$ is symmetrizable, there exists a diagonal matrix $\Gamma $ and a symmetric matrix $S={}^t S$ such that $C=\Gamma \cdot S\in M(n,n)$.
Since we can assume that $S$\rq{s} entries are real numbers, there exists an orthogonal matrix $P$ and a symmetric invertible matrix $Q\in M(l,l)$ such that 
\begin{align*}
S
=
{}^t P\cdot 
\left (
\begin{array}{c|c}
Q&O_{l,n-l}
\\
\hline 
O_{n-l,l}&O_{n-l,n-l}
\end{array}
\right )
\cdot P
=
P^{-1}\cdot 
\left (
\begin{array}{c|c}
Q&O_{l,n-l}
\\
\hline 
O_{n-l,l}&O_{n-l,n-l}
\end{array}
\right )
\cdot P.
\end{align*}
Put 
\begin{align*}
&
P=\left (
\begin{array}{c}
P_1
\\
\hline 
P_2
\end{array}
\right ),
\quad 
(P_1\in M(l,n), P_2\in M(n-l,n))&
\end{align*}
(Here, note that $P_1$ has $\rank P_1=l$ since $P$ is invertible.).
Then we have
\begin{align*}
{}^t P_1\cdot Q\cdot P_1=S.
\end{align*}
Here, we consider a pentad of Cartan type $P(l,n;Q,P_1,\Gamma )$.
Its Cartan matrix is 
\begin{align*}
&
\Gamma \cdot {}^t P_1\cdot Q\cdot P_1=\Gamma \cdot S =C.
&
\end{align*}
Thus, we have the following isomorphism
\begin{align*}
&
L(l,n; Q,P_1,\Gamma )\simeq (Z\oplus G^{\prime }(C))\oplus \Delta \simeq G^{\prime }(C)
&
\end{align*}
since $\dim Z=\rank P_1-\rank C=l-l=0$, $\dim \Delta =l- \rank P_1=l-l=0$.
Thus, the reduced contragredient Lie algebra $G^{\prime }(C)$ is isomorphic to the PC Lie algebra $L(l,n;Q,P_1,\Gamma )$ with symmetric bilinear form $B_Q$.
}
\item {
Let $n $ be the order of $C$, $l$ the rank of $C$, $C_{ij}$ the $(i,j)$-entry of $C$, i.e. $C=(C_{ij})_{i,j=1,\ldots ,n}$ is an $n\times n$ square matrix with rank $l$.
We take matrices $A_{12}\in M(n,n-l)$, $A_{21}\in M(n-l,n)$, $A_{22}\in M(n-l,n-l)$ arbitrarily so that the following matrix $A$ defined by
\begin{align*}
A
=
\left (
\begin{array}{c|c}
C&A_{12}
\\
\hline
A_{21}&A_{22}
\end{array}
\right )
\in M(2n-l,2n-l)
\end{align*}
becomes an invertible matrix.
If we put 
\begin{align*}
D=
\left (
\begin{array}{c}
E_n
\\
\hline 
O_{n-l,n}
\end{array}
\right )
\in M(2n-l,n)
,
\quad \Gamma =E_n,
\end{align*}
then the pentad of Cartan type $P(2n-l,n;A,D,\Gamma )$ has Cartan matrix
\begin{align}\label{cal;car1}
\Gamma \cdot {}^t D\cdot A\cdot D=E_n\cdot {}^t D\cdot A\cdot D=\left (
\begin{array}{c|c}
E_n
&
O_{n,n-l}
\end{array}
\right )
\left (
\begin{array}{c|c}
C&A_{12}
\\
\hline
A_{21}&A_{22}
\end{array}
\right )
\left (
\begin{array}{c}
E_n
\\
\hline 
O_{n-l,n}
\end{array}
\right )
=C.
\end{align}
Here, since $A$ is invertible, note that the row vectors of $\Gamma \cdot {}^t D \cdot A\in M(n,2n-l)$ are linearly independent.
So, if we put $h_i=[e_i,f_i]\in \h ^{2n-l}$ $(i=1,\ldots ,n)$ under the notations in Definition \ref {defn;1}, the vectors $\{h_1,\ldots ,h_n\}$ are linearly independent (see \cite [Corollary 2.26]{Sa5} or \cite [Lemma 3.1 (i)]{Sa6}).
Moreover, under the identification $\h ^{2n-r}=\{(a_1,\ldots ,a_{2n-r})\mid a_i\in \C \}$, we can identify the element $h_i$ $(i=1,\ldots ,n)$ with the $i$-th row vector of $\Gamma \cdot {}^t D\cdot A=\left (\begin{array}{c|c}C&A_{12}\end{array}\right )$ (see \cite [Proposition 2.11, Definition 2.12]{Sa5}).
Thus, we have a realization of $C$ as
$$
(\h^{2n-l}, \{\alpha _1,\ldots ,\alpha _n\}, \{h_1,\ldots ,h_n\}),
$$ 
where $\alpha _j \in (\h ^{2n-l})^*$ is a linear function defined by $\alpha _j(a_1,\ldots ,a_{2n-l})=a_j$.
(Note that $\dim \h ^{2n-l}=2n-l=2(\text {the order of $C$})-\rank C$).
Let us prove that $L(2n-l,n;A,D,\Gamma )$ is isomorphic to the Kac-Moody algebra $\g (C)$.
Take a principal gradation (see \cite [\S 1.5]{Kac1})
$$
\g (C)=\bigoplus _{j\in \Z}\g _j (\mathbf {1} ),\quad \g _0(\mathbf {1})=\h ^{2n-l}, \quad \g _{-1}(\mathbf {1})=\sum _{i=1,\ldots ,n}\C f_i, \quad \g _1(\mathbf {1})=\sum _{i=1,\ldots ,n}\C e_i. 
$$
Since $L(2n-l,n;A,D,\Gamma )$ is a minimal graded Lie algebra (see \cite [the proof of Theorem 1.3]{Sa4}) \footnote {Although we assume the symmetricity of a bilinear form in the statement of \cite [Theorem 1.3]{Sa4}, it does not need to show that $L(2n-l,n;A,D,\Gamma )$ is minimal.} with local part 
\begin{align}
\left (\sum _{i=1,\ldots ,n}\C f_i\right )\oplus \h ^{2n-l}\oplus \left (\sum _{i=1,\ldots ,n}\C e_i\right )\simeq \g _{-1}(\mathbf {1})\oplus \g _{0}(\mathbf {1})\oplus \g _{1}(\mathbf {1}),\label{eq;eqkac}
\end{align}
there exists an epimorphism 
$$
\phi \colon \g (C)\rightarrow L(2n-l,n;A,D,\Gamma ), \quad \phi \mid _{\g _{-1}(\mathbf {1})\oplus \g _{0}(\mathbf {1})\oplus \g _{1}(\mathbf {1})}=\mathrm {id}
$$
(see \cite [Definition 6, the definition of minimal graded Lie algebras]{Kac2}).
Thus, $\g (C)/\Ker \phi \simeq L(2n-l,n;A,D,\Gamma )$.
Since $\Ker \phi \cap \h ^{2n-l}=\{0\}$, we have that $\Ker \phi =\{0\}$ from the definition of $\g (C)$.
Thus, we have an isomorphism $\g (C)\simeq L(2n-l,n;A,D,\Gamma )$.
}
\item {To show this claim, it is sufficient to show that the derived algebra 
$
[L(2n-l,n;A,D,\Gamma ),L(2n-l,n;A,D,\Gamma )]
$
defined above is isomorphic to $G(C)$.
Note that the derived Lie algebra
$
[L(2n-l,n;A,D,\Gamma ),L(2n-l,n;A,D,\Gamma )]
$
has a canonical structure of grading as a subalgebra of $L(2n-l,n;A,D,\Gamma )$.
In this grading, the local part is 
\begin{align*}
&
\left (\sum _{i=1,\ldots ,n}\C f_i\right )\oplus \left (\sum _{i=1,\ldots ,n}\C [e_i,f_i]\right )\oplus \left (\sum _{i=1,\ldots ,n}\C e_i\right )
&
\\
&
\quad 
\simeq 
\g _{-1}(\mathbf {1})\oplus \left (\sum _{i=1,\ldots ,n}\C [e_i,f_i]\right )\oplus \g _{1}(\mathbf {1})
\quad 
\left (
\subset \g _{-1}(\mathbf {1})\oplus \g _{0}(\mathbf {1})\oplus \g _{1}(\mathbf {1})
\right ),
&
\end{align*}
via the isomorphism (\ref{eq;eqkac}).
In fact, if not, it must exists an $e_i$ or $f_i$ for some $i$ such that $[h,e_i]=0$ or $[h,f_i]=0$ for any $h\in \h^{2n-l}$.
However, it contradicts to the assumption that $A
=
\left (
\begin{array}{c|c}
C&A_{12}
\\
\hline
A_{21}&A_{22}
\end{array}
\right )$
is invertible.
Since $\g_j$ $(j\geq 1)$ (respectively $j\leq -1$) is generated by $\g _1$ (respectively $\g _{-1}$) (see \cite [Definitions 2.8, 2.12]{Sa3}), we have an isomorphism
\begin{align}
&
[L(2n-l,n;A,D,\Gamma ),L(2n-l,n;A,D,\Gamma )]\simeq \left (\bigoplus _{j\neq 0}\g _j (\mathbf {1} )\right )\oplus \left (\sum _{i=1,\ldots ,n}\C [e_i,f_i]\right )
\label{iso;dergraded}
&
\\
&
\quad 
\left (
\subset \bigoplus _{j\in \Z}\g _j (\mathbf {1} )=\g (C)\simeq L(2n-l,n;A,D,\Gamma )
\right ).\notag
&
\end{align}
The derived Lie algebra $[L(2n-l,n;A,D,\Gamma ),L(2n-l,n;A,D,\Gamma )]$ is minimal.
Indeed, since $L(2n-l,n;A,D,\Gamma )$ is transitive, the derived Lie algebra $[L(2n-l,n;A,D,\Gamma ),L(2n-l,n;A,D,\Gamma )]$ is also transitive.
It means the derived Lie algebra $[L(2n-l,n;A,D,\Gamma ),L(2n-l,n;A,D,\Gamma )]$ is minimal (see \cite [p.1278 Proposition 5]{Kac2}).
Thus, from (\ref {cal;car1}), we have that the graded Lie algebra (\ref {iso;dergraded}) is isomorphic to the contragredient Lie algebra whose Cartan matrix is $C$.
This finishes the proof.
}
\end{enumerate}
\end{proof}
From Theorem \ref{maintheo1}, we can say that the class of Kac-Moody algebras is contained in the class of PC Lie algebras.
\par 

\section {A construction of ``HUGE'' Lie algebra which contains all finite-dimensional representations of $\sl _2$}
Using the notion of PC Lie algebras, we can construct a ``LARGE'' Lie algebra from a given Kac-Moody algebra and its representation, moreover, we can explicitly express such Lie algebra under the assumption of Theorem \ref{theo;pcchain}.
For example, a ``LARGE'' Lie algebra which is constructed by $\sl _2 $ and its $(n+1)$-dimensional irreducible representation is isomorphic to 
$$
G\left (
\left (
\begin{array}{cc}
2&-n\\-n&n^2/2
\end{array}
\right )
\right )
$$
(see \cite [Example 3.7]{Sa6}).
In the remain part of this paper, we shall extend this example.
\subsection {some properties and notations of $\sl _2$ and its representations}
Here, let us consider a graded Lie algebra which contains Lie algebra $\sl _2$ and its ``infinitely many'' representations.
First, let us recall some fundamental properties of $\sl _2$ and its representations (defined over $\C$):
\begin{itemize}
\item {the Lie algebra $\sl _2 $ is simple},
\item {thus, any finite-dimensional representation of $\sl _2$ is completely reducible},
\item {for any positive integer $n$, there exists an irreducible representation whose dimension is $n$}, 
\item {such irreducible representation is determined uniquely up to equivalence}.
\end{itemize}
Next, we shall introduce the following notations used in the remain part of this paper.
\footnote {
We import these notations from the theory of prehomogeneous vector spaces.
}
\begin{defn}
We let $\g =\sl _2$.
We denote the $(n+1)$-dimensional irreducible representation of $\sl _2$ by $n\Lambda _1$ up to equivalence.
We denote the representation space of $n\Lambda _1$ by $V(n+1)$.
These spaces $V(n+1)$ $(n=0,1,2,\ldots )$ has dimension $n+1$, i.e., $\dim V(n+1)=n+1$.
\end{defn}
\begin{ex}
We denote by $0\Lambda _1$ a representation which is equivalent to the trivial representation on $1$-dimensional vector space.
We denote by $1\Lambda _1$ a representation which is equivalent to the canonical representation of $\sl _2$ on the $2$-dimensional vector space of column vectors.
We denote by $2\Lambda _1$ the representation which is equivalent to the representation of $\sl _2$ on $\mathrm {Sym} _2$, where $\mathrm {Sym} _2$ is the $3$-dimensional space consisted by the space of all symmetric matrices of size $2$, defined by
\begin{align*}
&\sl _2\curvearrowright \mathrm {Sym} _2\rightarrow \mathrm {Sym} _2
\\
&A\curvearrowright X\mapsto AX+X{}^tA.
\end{align*}
\end{ex}
\begin{ex}
The following representations $\rho, \sigma $ of $\sl _2$ on $\mathrm {M} _2$, where $\mathrm {M}_2$ is the space of all matrices of size $2$, defined by
\begin{align*}
&
\rho \colon 
\begin{cases}
&\sl _2\curvearrowright \mathrm {M} _2\rightarrow \mathrm {M} _2
\\
&A\curvearrowright X\mapsto AX+X{}^tA
\end{cases}
&
&
\sigma \colon 
\begin{cases}
&\sl _2\curvearrowright \mathrm {M} _2\rightarrow \mathrm {M} _2
\\
&A\curvearrowright X\mapsto AX
\end{cases}
&
\end{align*}
are respectively denoted by $\rho \simeq 2\Lambda _1+0\Lambda _1$, $\sigma \simeq 1\Lambda _1\oplus 1\Lambda _1$.
\end{ex}
The following claim on ``highest weight'' is well-known.
We can easily check that a similar claim holds on ``lowest weight''.
Here, we let $(y,h,x)$ be a basis of $\sl _2$ satisfying 
\begin{align*}
&
[h,y]=-2y,
&
&
[h,x]=2x,
&
&
[x,y]=h.
&
\end{align*}
\begin{pr}
\label{pr;sl2lowestweightvector}
There exists a ``highest'' (respectively ``lowest'') weight vector $v^{n\Lambda _1}\in V(n+1)$ $(\text {respectively }v_{n\Lambda _1}\in V(n+1))$ which satisfies the following 2 conditions:
\begin{itemize}
\item{
$n\Lambda _1 (x\otimes v^{n\Lambda _1})=0\qquad (\text {respectively }n\Lambda _1(y\otimes v_{n\Lambda _1})=0),$
}
\item {
the representation space $V(n+1)$ is generated by $v^{n\Lambda _1}$ and $h,y\in \sl_2$ (respectively $v_{n\Lambda _1}$ and $h,x\in \sl_2$).
}
\end{itemize}
\end{pr}
\begin{pr}
The highest (respectively lowest) weight vector $v^{n\Lambda _1}\in V(n+1)$ (respectively $v_{n\Lambda _1}\in V(n+1)$) satisfies
\begin{align*}
n\Lambda _1(h\otimes v^{n\Lambda _1})=nv^{n\Lambda _1}
\qquad (\text {respectively }n\Lambda _1(h\otimes v_{n\Lambda _1})=-nv_{n\Lambda _1}).
\end{align*}
\end{pr}
\begin{proof}
We can easily check them by easy calculations.
\end{proof}

\subsection {all finite-dimensional representation of $\sl _2$ and a graded Lie algebra}
We shall consider a graded Lie algebra which contains ``infinitely many'' $n\Lambda _1$'s.
\begin{pr}\label {pr;sl2_fd_mod}
Any finite-dimensional representation of $\sl _2$ can be written in the form of 
$$
n_1\Lambda _1\oplus n_2\Lambda _1\oplus \cdots \oplus n_k\Lambda _1
$$
for some $k\geq 1$, $n_1,\ldots ,n_k\in \N \cup \{0\}$ (the integers $n_i$'s are NOT necessarily $n_i\neq n_j$ $(i\neq j)$).
\end{pr}
\begin{proof}
This claim is immediate.
\end{proof}

\begin{defn}
We define a representation $\rho _{\mathrm {f.d.}}$ of $\sl _2$ by
\begin{align*}
\rho _{\mathrm {f.d.}}
&=(0\Lambda _1\oplus 0\Lambda _1\oplus \cdots )\oplus (1\Lambda _1\oplus 1\Lambda _1\oplus \cdots )\oplus \cdots \oplus (n\Lambda _1\oplus n\Lambda _1\oplus \cdots )\oplus \cdots 
\\
& =\bigoplus _{n=0}^{\infty }\left (\text {the direct sum of countable many copies of $n\Lambda _1$}\right ),
\end{align*}
and denote the representation space of $\rho _{\mathrm {f.d.}}$ by $U_{\mathrm {f.d.}}$.
\end{defn}
\begin{pr}
We have the following claims.
\begin{enumerate}
\item {
Any finite-dimensional representation of $\sl _2$ is a subrepresentation of $\rho _{\mathrm {f.d.}}$.
}
\item {
The representation space $U_{\mathrm {f.d.}}$ is canonically regarded as a submodule of $\Hom (U_{\mathrm {f.d.}},\C)$.
}
\end{enumerate}
\end{pr}
\begin{proof}
\begin{enumerate}
\item {This claim is immediate.}
\item {Since the dual space $\Hom (V(n+1),\C)$ of $n\Lambda _1$ $(n=0,1,2,\ldots )$ is isomorphic to $n\Lambda _1$, we have a natural pairing 
$$
U_{\mathrm {f.d.}}\times U_{\mathrm {f.d.}}\rightarrow \C.
$$
Thus, $U_{\mathrm {f.d.}}$ is canonically identified with a submodule of $\Hom(U_{\mathrm {f.d.}},\C)$.
}
\end{enumerate}
\end{proof}
Under these, let us consider a graded Lie algebra which contains the Lie algebra $\sl _2$ and its representation $U_{\mathrm {f.d.}}$.
\begin{defn}
We have a standard pentad $(\sl _2,\rho _{\mathrm {f.d.}}, U_{\mathrm {f.d.}}, U_{\mathrm {f.d.}}\ (\simeq \Hom (U_{\mathrm {f.d.}}, \C)), K_{\sl _2})$, where the bilinear form $K_{\sl _2}$ on $\sl _2$ is the Killing form\footnote {For the reason why this pentad is standard, see \cite [Lemma 2.3]{Sa3}.}.
We put
\begin{align*}
L(\sl _{2},\mathrm {f.d.}):=L(\sl _2,\rho _{\mathrm {f.d.}}, U_{\mathrm {f.d.}}, U_{\mathrm {f.d.}}, K_{\sl _2}), 
\end{align*}
i.e., $L(\sl _{2},\mathrm {f.d.})$ is the corresponding Lie algebra to the standard pentad $(\sl _2,\rho _{\mathrm {f.d.}}, U_{\mathrm {f.d.}}, U_{\mathrm {f.d.}}, K_{\sl _2})$.
\end{defn}
Let us consider the structure of this Lie algebra $L(\sl _{2},\mathrm {f.d.})$.
This Lie algebra $L(\sl _{2},\mathrm {f.d.})$ has a canonical structure of a graded Lie algebra:
\begin{align}
L(\sl _{2},\mathrm {f.d.})= \bigoplus _{n\in \Z }L_n
\label {sl2fd_1}
\end{align}
such that 
\begin{align*}
&
L_0\simeq \sl _2\qquad (\text {as Lie algebras}),
&
\\
&
(L_0,\ad,L_1)\simeq (\sl _2,\rho _{\mathrm {f.d.}}, U_{\mathrm {f.d.}}),
\quad 
(L_0,\ad,L_{-1})\simeq (\sl _2,\rho _{\mathrm {f.d.}}, U_{\mathrm {f.d.}})
\quad (\text {as $L_0\simeq \sl_2$-modules}).
&
\end{align*}
From the definition of $L(\sl _{2},\mathrm {f.d.})$, we have the following proposition immediately.
\begin{pr}
Let $(\sigma ,U)$ be an arbitrary finite-dimensional representation of $\sl_2$.
Then $L(\sl _{2},\mathrm {f.d.})=\bigoplus _{n\in \Z }L_n$ has an $L_0$-submodule $W\subset L_1$ such that $(\sigma ,U)\simeq (\ad,W)$ as $\sl _2\simeq L_0$-modules.
\end{pr}
That is, $L(\sl _{2},\mathrm {f.d.})$ is a graded Lie algebra which contains ``all'' finite-dimensional representation of $\sl _2$.
Here, we have a natural question that ``how to express the structure of $L(\sl _{2},\mathrm {f.d.})$?''.
\par
We can answer this question using chain rule (\cite [Theorem 3.26]{Sa3}) as follows.
\begin{defn}\label {defn;Chrep}
Let $(y,h,x)$ be a basis of $\sl _2$ satisfying 
\begin{align*}
&[h,y]=-2y,&
&[h,x]=2x,&
&[x,y]=h.&
\end{align*}
Let $\h =\C h $ be a Cartan subalgebra of $\sl _2$.
We put symbols $\{e_{i,j}^+\}_{i,j\in \N \cup \{0\}}$, $\{e_{i,j}^-\}_{i,j\in \N \cup \{0\}}$, and we set vector spaces 
$$
\bigoplus _{i,j\in \N \cup \{0\}}\C e_{i,j}^+
=
\bigoplus _{i=0}^{\infty }
\left(
\bigoplus _{j=0}^{\infty }
\C e_{i,j}^+
\right ),
\qquad 
\bigoplus _{i,j\in \N \cup \{0\}}\C e_{i,j}^-
=
\bigoplus _{i=0}^{\infty }
\left(
\bigoplus _{j=0}^{\infty }
\C e_{i,j}^-
\right ).
$$
We define representations $\rho _{\h}^{+}$, $\rho _{\h}^{-}$ of $\h$ on $\bigoplus _{i,j\in \N \cup \{0\}}\C e_{i,j}^+$  and on $\bigoplus _{i,j\in \N \cup \{0\}}\C e_{i,j}^-$ by 
$$
\rho _{\h}^{+}(h\otimes e_{i,j})=ie_{i,j}^+,
\qquad
\rho _{\h}^{-}(h\otimes e_{i,j})=-ie_{i,j}^-.
$$
\end{defn}
\begin{pr}\label {pr;sl_2_univ_0}
We retain to use the notations in Definition \ref {defn;Chrep}.
The Lie algebra $L(\sl _{2},\mathrm {f.d.})$ is isomorphic to 
\begin{align}\label{eq;sl2_univ_0}
L\left (\C h, \ad \oplus \rho _{\h}^{-}, \left (\C x \oplus \bigoplus _{i,j\in \N \cup \{0\}}\C e_{i,j}^-\right ), \left (\C y\oplus \bigoplus _{i,j\in \N \cup \{0\}}\C e_{i,j}^+\right ), K_{\sl _2}\mid _{\C h }\right )
\end{align}
up to grading.
Here, a pairing $\langle \cdot, \cdot \rangle$ between  $\left (\C x\oplus \bigoplus _{i,j\in \N \cup \{0\}}\C e_{i,j}^-\right )$ and $\left (\C y \oplus \bigoplus _{i,j\in \N \cup \{0\}}\C e_{i,j}^+\right )$ is defined by
\begin{align*}
&\langle x,e_{i,j}^+\rangle =\langle e_{i,j}^-,y\rangle =0,& 
&\langle x,y\rangle =4,& 
&\langle e_{i,j}^-,e_{k,l}^+\rangle =4\delta _{(i,j),(k,l)}&
&(\text {for all $i,j,k,l\in \N\cup \{0\}$}).&
\end{align*}
\end{pr}
\begin{proof}
It is easy to show that this pentad is standard.
Indeed, the condition 1 in Definition \ref{defn;stap} is immediate.
The condition 2 follows from \cite [Lemma 2.3]{Sa3}.
\par 
The Lie algebra $\sl _2$ is obtained by a standard pentad
$$
\sl _2 \simeq L(\C h, \ad ,\C x,\C y, K_{\sl _2}\mid _{\C h })
$$
(see \cite [Example 3.14]{Sa5}).
Under this isomorphism, for any $n\in \N\cup \{0\}$, the representation $(n\Lambda _1, V(n+1))$ and its dual $(n\Lambda _1, \Hom (V(n+1),\C))\simeq (n\Lambda _1, V(n+1))$ are respectively isomorphic to the positive and negative extensions (the notion of positive/negative extension is defined in \cite [Theorems 3.12, 3.14]{Sa3})
$$
V(n+1)\simeq \widetilde{\left ( \C v_{n\Lambda_1} \right )}^+ ,
\qquad 
\Hom (V(n+1),\C)\simeq \widetilde{\left ( \C v^{n\Lambda_1} \right )} ^-
$$
with respect to the pentad $(\C h, \ad ,\C x,\C y, K_{\sl _2}\mid _{\C h })$ (see \cite [Theorem 3.17]{Sa3}).
Thus, applying \cite [Theorem 3.26]{Sa3}, we  have our result.
\end{proof}
A PC Lie algebra can be expressed by matrices.
Similarly, the Lie algebra $L(\sl _{2},\mathrm {f.d.})\simeq (\ref {eq;sl2_univ_0})$ can be also expressed by matrices as follows.
\begin{defn}
We define a lexicographical order $\prec$ on $\{-1\}\cup (\N\cup \{0\})^2$ as:
\begin{align*}
&-1\prec \alpha \quad \text {for all $\alpha \in (\N\cup \{0\})^2$},&
\\
&(\alpha _1,\alpha _2)\prec (\beta _1,\beta _2) \text { if and only if }
\begin{cases}
&
\alpha _1<\beta _1
\\
&\qquad \text {or}
\\
&
\alpha _1=\beta _1\text{ and }\alpha _2<\beta _2
\end{cases}
\quad \text {for all $\alpha _i,\beta _i\in \N\cup \{0\}$}.
&
\end{align*}
\end{defn}

\begin{defn}
We put matrices
\begin{align*}
&\widetilde {A}=\begin{pmatrix}\frac{1}{8}\end{pmatrix}_{1\times 1},&
\\
&\widetilde {D}=
\left (\begin{array}{c|c|c|c|c|c}{2}&\begin{array}{cc}0&\cdots \end{array}&\begin{array}{cc}-1&\cdots \end{array}&\cdots &\begin{array}{cc}-n&\cdots \end{array}&\cdots\end{array}
\right )
_{1\times (\{-1\}\cup(\N\cup \{0\})^2)},&
\\
&\widetilde {\Gamma }=\begin{pmatrix}4\delta _{ij}\end{pmatrix}_{i,j\in \{-1\}\cup (\N \cup \{0\})^2}
=
\left (
\begin{array}{c|c|c|c}
4&\begin{array}{cc}0&\cdots \end{array}&\begin{array}{cc}0&\cdots \end{array}&\cdots 
\\ \hline 
\begin{array}{c}0\\ \vdots \end{array}&\begin{array}{cc}4&\\&\ddots  \end{array}&\begin{array}{cc}0&\cdots \\ \vdots &\ddots  \end{array}&\cdots 
\\ \hline
\begin{array}{c}0\\ \vdots \end{array}&\begin{array}{cc}0&\cdots \\ \vdots &\ddots  \end{array}&\begin{array}{cc}4&\\&\ddots  \end{array}&\cdots 
\\ \hline 
\vdots &\vdots &\vdots &\ddots 
\end{array}
\right )_{(\{-1\}\cup (\N \cup \{0\})^2)\times (\{-1\}\cup (\N \cup \{0\})^2)}
,&
\end{align*}
That is, 
\begin{itemize}
\item {$\widetilde {A}$ is a $1\times 1$-matrix with entry $1/8$},
\item {$\widetilde {D}=(\widetilde {d}_{1,\alpha })_{1\times (\{-1\}\cup(\N\cup \{0\})^2)}$ is a $1\times (\{-1\}\cup(\N\cup \{0\})^2)$-matrix whose entries are given as
$$
\widetilde {d}_{1,\alpha }
=
\begin{cases}
2& (\alpha =-1)
\\
-i& (\alpha =(i,j),\ i,j\in \N\cup \{0\})
\end{cases}
$$ 
and ones are aligned as the lexicographical order $\prec$},
\item {$\widetilde {\Gamma }$ is $4$ multiple of the $\text {``unit matrix''}$ with order $\{-1\}\cup(\N\cup \{0\})^2$}.
\end{itemize}
\end{defn}
\begin{defn}
We put a matrix $\widetilde {C}$ as
\begin{align*}
&
\widetilde {C}=(\widetilde {C}_{\alpha ,\beta })_{\left (\{-1\}\cup \left(\N\cup \{0\}\right )^2\right )\times \left (\{-1\}\cup \left(\N\cup \{0\}\right )^2\right )}
=\widetilde {\Gamma }\cdot {}^t \widetilde {D}\cdot \widetilde {A}\cdot \widetilde {D}
\\
&
=\left (
\begin{array}{c|c|c|c}
4&\begin{array}{cc}0&\cdots \end{array}&\begin{array}{cc}0&\cdots \end{array}&\cdots 
\\ \hline 
\begin{array}{c}0\\ \vdots \end{array}&\begin{array}{cc}4&\\&\ddots  \end{array}&\begin{array}{cc}0&\cdots \\ \vdots &\ddots  \end{array}&\cdots 
\\ \hline
\begin{array}{c}0\\ \vdots \end{array}&\begin{array}{cc}0&\cdots \\ \vdots &\ddots  \end{array}&\begin{array}{cc}4&\\&\ddots  \end{array}&\cdots 
\\ \hline 
\vdots &\vdots &\vdots &\ddots 
\end{array}
\right )
\left (\begin{array}{c}2\\ \hline \begin{array}{c}0\\ \vdots \end{array}\\ \hline \begin{array}{c}-1\\ \vdots \end{array}\\ \hline \vdots \\ \hline \begin{array}{c}-n\\ \vdots \end{array}\\ \hline \vdots\end{array}\right )
\begin{pmatrix}
\frac{1}{8}
\end{pmatrix}
\left (\begin{array}{c|c|c|c|c}2&\begin{array}{cc}0&\cdots \end{array}&\cdots &\begin{array}{cc}-n&\cdots \end{array}&\cdots\end{array}\right )
\\
&\quad =
\left (
\begin{array}{c|c|c|c|c|c}
2&\begin{array}{cc}0&\cdots \end{array}&\begin{array}{cc}-1&\cdots \end{array}&\cdots &\begin{array}{cc}-n&\cdots \end{array}&\cdots
\\ \hline 
\begin{array}{c}0\\ \vdots \end{array}&\begin{array}{cc}0&\cdots \\ \vdots &\ddots \end{array}&\begin{array}{cc}0&\cdots \\ \vdots &\ddots \end{array}&\cdots &\begin{array}{cc}0&\cdots \\ \vdots &\ddots \end{array}&\cdots 
\\ \hline
\begin{array}{c}-1\\ \vdots \end{array}&\begin{array}{cc}0&\cdots \\ \vdots &\ddots \end{array}&\begin{array}{cc}\frac{1}{2}&\cdots \\ \vdots &\ddots \end{array}&\cdots &\begin{array}{cc}\frac{n}{2}&\cdots \\ \vdots &\ddots \end{array}&\cdots 
\\ \hline
\vdots &\vdots &\vdots&\ddots &\vdots &\ddots 
\end{array}
\right )_{\left (\{-1\}\cup \left(\N\cup \{0\}\right )^2\right )\times \left (\{-1\}\cup \left(\N\cup \{0\}\right )^2\right )}
&
\end{align*}
(the multiple of matrices is defined canonically).
\end{defn}
That is, the matrix $\widetilde {C}=(\widetilde {C}_{\alpha ,\beta })$ is a $\left (\{-1\}\cup \left(\N\cup \{0\}\right )^2\right )\times \left (\{-1\}\cup \left(\N\cup \{0\}\right )^2\right )$-matrix with entries 
\begin{align*}
&
\widetilde {C}_{\alpha ,\beta }
=
\begin{cases}
2&(\alpha =\beta =-1)
\\
-i&\left ((\alpha ,\beta )=(-1,(i,j)), \ ((i,j),-1)\right )\quad \left (i,j\in \N\cup \{0\}\right )
\\
\frac{1}{2}ik&\left ((\alpha ,\beta )=((i,j),(k,l))\right )\quad (i,j,k,l\in \N\cup \{0\})
\end{cases}.
\end{align*}
\begin{theo}\label{theo;sl2_univ_9}
Under the notations of Definition \ref{defn;Chrep}, the Lie algebra (\ref{eq;sl2_univ_0}) has a minimal graded Lie algebra structure $(\ref {eq;sl2_univ_0})\simeq \bigoplus _{p\in \Z}\l_p$ with local part
\begin{align*}
&\l_{1}\simeq \C x \oplus \bigoplus _{i,j\in \N \cup \{0\}}\C e_{i,j}^-,&
&\l_0\simeq \C h,&
&\l_{-1}\simeq \C y\oplus \bigoplus _{i,j\in \N \cup \{0\}}\C e_{i,j}^+,&
\\
&[h,y]=-2y,&
&[h,x]=2x,&
&[x,y]=h,&
\\
&[h,e_{i,j}^-]=-ie_{i,j}^-,&
&[h,e_{i,j}^+]=ie_{i,j}^+,&
&&
\\
&[y,e_{i,j}^-]=0,&
&[x,e_{i,j}^+]=0,&
&[e_{i,j}^-,e_{k,l}^+]=\frac{-i}{2}\delta _{(i,j),(k,l)}h&
\end{align*}
for $i,j,k,l\in \N\cup \{0\}$.
\end{theo}
\begin{proof}
It suffices to check the equations 
\begin{align}
&[y,e_{i,j}^-]=[x,e_{i,j}^+]=0,\label {eq;bracket_univ_1}&
\\
&[e_{i,j}^-,e_{k,l}^+]=\frac{-i}{2}\delta _{(i,j),(k,l)}h.\label {eq;bracket_univ_2}&
\end{align}
Let us check the equation (\ref {eq;bracket_univ_1}).
From the definition of Lie algebras associated with a standard pentad, the elements 
$
[y,e_{i,j}^-],
[x,e_{i,j}^+]
$
are elements of $\l_0\simeq \C h$ satisfying 
$$
K\mid _{\sl _2}(h,[y,e_{i,j}^-])=\langle e_{i,j}^-,[h,y]\rangle =-2\langle  e_{i,j}^-,y\rangle ,
\qquad 
K\mid _{\sl _2}(h,[x,e_{i,j}^+])=\langle [h,x], e_{i,j}^+\rangle =2\langle x, e_{i,j}^+\rangle
$$
respectively.
Since $\langle e_{i,j}^-,y\rangle =\langle x,e_{i,j}^+\rangle =0$ (see Proposition \ref {pr;sl_2_univ_0}), we have the equation (\ref {eq;bracket_univ_1}).
\par 
The equation (\ref {eq;bracket_univ_2}) can be checked by a similar argument to one of (\ref {eq;bracket_univ_1}).
For any $i,j,k,l\in \N\cup \{0\}$, the element $[e_{i,j}^-,e_{k,l}^+]$ satisfies 
$$
K\mid _{\sl _2}(h,[e_{i,j}^-,e_{k,l}^+])=\langle [h,e_{i,j}^-],e_{k,l}^+\rangle =-i\langle e_{i,j}^-,e_{k,l}^+\rangle =-4i\delta _{(i,j),(k,l)}.
$$
Since 
$
K\mid _{\sl _2}(h,h)=8
$
and $[e_{i,j}^-,e_{k,l}^+]\in \C h$,
we have our claim.
\end{proof}
\begin{remark}
From (\ref {sl2fd_1}) and Theorems \ref {theo;sl2_univ_9}, we have two gradings of $L(\sl _{2},\mathrm {f.d.})$.
Although we have isomorphisms
$$
L(\sl _{2},\mathrm {f.d.})\simeq \bigoplus _{n\in \Z}L_n\simeq \bigoplus _{p\in \Z}\l_p\simeq L(\sl _{2},\mathrm {f.d.})
$$
as Lie algebras up to gradings, we have
$$
L(\sl _{2},\mathrm {f.d.})\simeq \bigoplus _{n\in \Z}L_n\not\simeq \bigoplus _{p\in \Z}\l_p\simeq L(\sl _{2},\mathrm {f.d.})
$$
as graded Lie algebras.
\end{remark}
As a corollary of Theorem \ref {theo;sl2_univ_9}, we have the following immediately.
\begin{theo}
Let $(M,\prec )$ be a non-empty finite subset of the ordered set $\left (\{-1\}\cup \left (\N \cup \{0\}\right )^2,\prec \right )$.
Let $\widetilde {C}^M=(\widetilde {C}^M_{\alpha, \beta })_{\alpha, \beta \in M}$ be the minor matrix of $\widetilde {C}$ corresponding to $M$.
Then the graded Lie algebra $L(\sl _2, \mathrm {f.d.})\simeq \bigoplus _{p\in \Z }\l _p$ with the grading defined in Theorem \ref{theo;sl2_univ_9} has a graded subalgebra $\l^M$ which is isomorphic to the reduced contragredient Lie algebra with Cartan matrix $\widetilde {C}^M$.
\end{theo}
\begin{proof}
First, let us recall how to construct the (reduced) contragredient Lie algebra from a given matrix (see \cite [pp.1279--1280]{Kac2}).
To obtain the contragredient Lie algebra $G(\widetilde {C}^M)$, we take symbols $\{f_{\alpha },h_{\alpha }, e_{\alpha }\}_{\alpha \in M}$ and put
\begin{align*}
&
G^M_{1}\simeq \bigoplus _{\alpha \in M}\C e_{\alpha },
&
&
G^M_0\simeq \bigoplus _{\alpha \in M}\C h_{\alpha },
&
&
G^M_{-1}\simeq \bigoplus _{\alpha \in M}\C f_{\alpha }.
&
\end{align*}
We put $\widehat {G^M}=G^M_{-1}\oplus G^M_{0}\oplus G^M_{1}$ and induce a local Lie algebra structure in $\widehat {G^M}$ as
\begin{align*}
&
[h_{\alpha },h_{\beta }]=0,
&
&
[e_{\alpha },f_{\beta }]=\delta _{\alpha,\beta }h_{\alpha }
&
&
[h_{\alpha },e_{\beta }]=\widetilde {C}^M_{\alpha, \beta }e_{\beta },
&
&
[h_{\alpha },f_{\beta }]=-\widetilde {C}^M_{\alpha, \beta }f_{\beta },
&
&
(\alpha, \beta \in M).
&
\end{align*} 
Then the minimal graded Lie algebra with local part $\widehat {G^M}$ is the contragredient Lie algebra $G(\widetilde {C}^M)$.
We denote the center of $G(\widetilde {C}^M)$ by $Z(\widetilde {C}^M)$.
Then the reduced contragredient Lie algebra $G^{\prime }(\widetilde {C}^M)$ is given by
$
G^{\prime }(\widetilde {C}^M)=G(\widetilde {C}^M)/Z(\widetilde {C}^M).
$
\par
To prove our claim, we shall give an injective homomorphism between $G^{\prime }(\widetilde {C}^M)$ and $L(\sl _2, \mathrm {f.d.})$.
We define a map $\phi $ between their local parts $\phi \colon \widehat {G^M}=G^M_{-1}\oplus G^M_{0}\oplus G^M_{1}\rightarrow \l_{-1}\oplus \l_0\oplus \l_1$ by:
\begin{align*}
&
f_{\alpha }
\overset {\phi }{\mapsto }
\begin{cases}
y&(\text {if }\alpha =-1)
\\
e_{i,j}^+&(\text {if }\alpha  =(i,j)\in (\N \cup \{0\})^2)
\end{cases},
&
&
h_{\alpha }
\overset {\phi }{\mapsto }
\begin{cases}
h&(\text {if }\alpha =-1)
\\
\frac{-i}{2}h&(\text {if }\alpha  =(i,j)\in (\N \cup \{0\})^2)
\end{cases},
&
\\
&
e_{\alpha }
\overset {\phi }{\mapsto }
\begin{cases}
x&(\text {if }\alpha =-1)
\\
e_{i,j}^-&(\text {if }\alpha  =(i,j)\in (\N \cup \{0\})^2)
\end{cases},
&
&
(\text {for $\alpha \in M$}).
&
\end{align*}
We can easily check that this map $\phi$ is a homomorphism of local Lie algebras.
The kernel of $\phi $ coincides with $Z(\widehat {C ^M})$.
In fact, since we can easily show that
\begin{align*}
\Ker (\phi )
=
\left \{
\sum _{\alpha \in \{-1\}\cup (\N\cup \{0\})^2}s_{\alpha }h_{\alpha }
\middle |
\begin{array}{c}
s_{\alpha }\in \C,
\\
s_{\alpha}=0\text { when $\alpha \not\in M$},
\\
s_{-1}+\sum _{(i,j)\in (\N\cup \{0\})^2}\frac{-is_{(i,j)}}{2}=0
\end{array}
\right \},
\end{align*}
we have an equation $Z(\widehat {C ^M})= \Ker (\phi )$ easily (see \cite [p.1280, Lemma 1]{Kac2}).
\par
To complete the proof, let us check that the graded subalgebra of $L\left (\sl _2, \mathrm {f.d.}\right )$ generated by $\phi (\widehat {G^M})$ is minimal.
If not, there exists a non-zero element $X\neq 0$ generated by $\{e_{\alpha }^-\}_{\alpha \in M}$ or $Y\neq 0$ generated by $\{e_{\alpha }^+\}_{\alpha \in M}$ (here, we denote $x$ and $y$ by $e_{-1}^-$ and $e_{-1}^+$ respectively) such that 
\begin{align*}
&[X,e_{\alpha }^+]=0,&
&[Y,e_{\alpha }^-]=0,&
&(\text {for any $\alpha \in M$})&
\end{align*}
respectively.
We can see that such elements do not exist from the followings:
\begin{itemize}
\item {
$[X,e_{\alpha }^+]=0$, $[Y,e_{\alpha }^-]=0$ for any $\alpha \in \{-1\}\cup (\N \cup \{0\})^2\setminus M$ from equations (\ref {eq;bracket_univ_1}), (\ref {eq;bracket_univ_2}) and the assumption that $X$, $Y$ are generated by elements $e_{\alpha }^-$, $e_{\alpha }^+$ ($\alpha \in M$) respectively,
} 
\item {
in general, a Lie algebra associated with a standard pentad has the transitivity for $|p|\geq 2$ (see \cite [Definitions 2.9, 2.12, the general construction of Lie algebras associated with a standard pentad]{Sa3}), and thus, the graded Lie algebra $L\left (\sl _2, \mathrm {f.d.}\right )\simeq \bigoplus _{p\in \Z}\l_p$ has transitivity.
}
\end{itemize}
\par 
Thus, the graded subalgebra of $L\left (\sl _2, \mathrm {f.d.}\right )\simeq \bigoplus _{p\in \Z}\l_p$ generated by $\phi (\widehat {G^M})$ is minimal and has local part which is isomorphic to $\widehat {G^M}/\Ker (\phi )\simeq \widehat {G^M}/Z(\widehat {C ^M})$.
It means that the graded Lie algebra $L\left (\sl _2, \mathrm {f.d.}\right )\simeq \bigoplus _{p\in \Z}\l_p$ has a subalgebra which is isomorphic to $G^{\prime }(\widetilde {C}^M)$.
This completes the proof.
\end{proof}
Using the general theory of standard pentads, we can construct a representation of $L\left (\sl _2, \mathrm {f.d.}\right )$ from any finite-dimensional representation of $\sl _2$.
Moreover, from $L\left (\sl _2, \mathrm {f.d.}\right )$ and such a representation, we can construct a Lie algebra.
However, it is NOT ``LARGER'' than $L\left (\sl _2, \mathrm {f.d.}\right )$.
\begin{theo}\label {th;sl2fd_closed}
Take an arbitrary finite-dimensional representation $(\sigma, U)$ of $\sl _2$.
Then the Lie algebra $L\left (\sl _2, \mathrm {f.d.}\right )=\bigoplus _{n\in \Z }L _n$ with grading defined in (\ref{sl2fd_1}) has representations $(\widetilde {\sigma }^+,\widetilde {U}^+=\bigoplus _{m\geq 0}U_m^+)$, $(\widetilde {\sigma }^-,\widetilde {U}^-=\bigoplus _{m\leq 0}U_m^-)$ such that
\begin{itemize}
\item {$(\widetilde {\sigma }^+\mid _{L_0},U_0^+)\simeq (\sigma ,U)$, $(\widetilde {\sigma }^-\mid _{L_0},U_0^-)\simeq (\sigma ,U)$ under the isomorphism $L_0\simeq \sl_2$},
\item {$(\widetilde {\sigma }^+,\widetilde {U}^+=\bigoplus _{m\geq 0}U_m^+)$, $(\widetilde {\sigma }^-,\widetilde {U}^-=\bigoplus _{m\leq 0}U_m^-)$ are transitive (we use the term ``transitive representation'' in the sense of \cite [Definition 1.3]{Sa3})},
\item {$\widetilde {U}^+$ (respectively $\widetilde {U}^-$) is generated by $L_0, L_1, U_0^+$ (respectively $L_0, L_{-1}, U_0^-$)}.
\end{itemize}
Moreover, if we denote the bilinear form on $L\left (\sl _2, \mathrm {f.d.}\right )$ which is obtained by expanding the Killing form of $\sl _2$ by $\widetilde {K}$ (see \cite [Peoposition 2.18]{Sa3}), then we have a standard pentad 
$$
\left (L\left (\sl _2, \mathrm {f.d.}\right ), \widetilde {\sigma }^+, \widetilde {U}^+,\widetilde {U}^-,\widetilde {K} \right )
$$
and an isomorphism 
\begin{align}
L\left (L\left (\sl _2, \mathrm {f.d.}\right ), \widetilde {\sigma }^+, \widetilde {U}^+,\widetilde {U}^-,\widetilde {K} \right )\simeq L\left (\sl _2, \mathrm {f.d.}\right )\label {eq;cld}
\end{align}
up to gradings.
\end{theo}
\begin{proof}
The existence of the representations $\widetilde {\sigma }^+$, $\widetilde {\sigma }^-$ follows from \cite [Theorems 3.12, 3.14, 3.17]{Sa3}.
Let us show the equation (\ref {eq;cld}).
Since we have a standard pentad
$(\sl _2, \sigma ,U,\Hom (U,\C),K_{\sl _2})$ (see \cite [Lemma 2.3]{Sa3}), we have an isomorphism of Lie algebras
\begin{align*}
&
L\left (L\left (\sl _2, \mathrm {f.d.}\right ), \widetilde {\sigma }^+, \widetilde {U}^+,\widetilde {U}^-,\widetilde {K} \right )
\simeq 
L\left (L(\sl _2,\rho _{\mathrm {f.d.}}, U_{\mathrm {f.d.}}, U_{\mathrm {f.d.}}, K_{\sl _2}), \widetilde {\sigma }^+, \widetilde {U}^+,\widetilde {U}^-,\widetilde {K} \right )
&
\\
&
\simeq L\left (L(\sl _2,\rho _{\mathrm {f.d.}}, U_{\mathrm {f.d.}}, U_{\mathrm {f.d.}}, K_{\sl _2}), \widetilde {\sigma }^+, \widetilde {U}^+,\widetilde {\Hom (U,\C)}^-,\widetilde {K} \right )\qquad (\text {from $U\simeq \Hom (U,\C)$})
&
\\
&
\simeq L\left (\sl _2, \rho _{\mathrm {f.d.}}\oplus \sigma , U_{\mathrm {f.d.}}\oplus U, U_{\mathrm {f.d.}}\oplus \Hom (U,\C), K_{\sl _2} \right )\qquad (\text {\cite [Theorem 3.26]{Sa3}})
&
\\
&
\simeq L(\sl _2,\rho _{\mathrm {f.d.}}, U_{\mathrm {f.d.}}, U_{\mathrm {f.d.}}, K_{\sl _2})\qquad (\text {from the definitions of $\rho _{\mathrm {f.d.}}, U_{\mathrm {f.d.}}$})
&
\\
&
\simeq L\left (\sl _2, \mathrm {f.d.}\right )
&
\end{align*}
up to gradings.
\end{proof}
\begin{col}
If we denote the canonical grading of $L\left (L\left (\sl _2, \mathrm {f.d.}\right ), \widetilde {\sigma }^+, \widetilde {U}^+,\widetilde {U}^-,\widetilde {K} \right )$ by 
$$
L\left (L\left (\sl _2, \mathrm {f.d.}\right ), \widetilde {\sigma }^+, \widetilde {U}^+,\widetilde {U}^-,\widetilde {K} \right )=\bigoplus _{n\in \Z}\L_n,
$$ 
this graded Lie algebra satisfies 
\begin{align*}
L\left (\sl _2, \mathrm {f.d.}\right )\simeq \L_0\subsetneq \bigoplus _{n\in \Z}\L_n\simeq L\left (\sl _2, \mathrm {f.d.}\right )
\end{align*}
from the equation (\ref {eq;cld}).
That is, the graded Lie algebra $L\left (L\left (\sl _2, \mathrm {f.d.}\right ), \widetilde {\sigma }^+, \widetilde {U}^+,\widetilde {U}^-,\widetilde {K} \right )$ is isomorphic to its proper Lie subalgebra.
\end{col}
Concluding the above argument roughly, we can obtain the followings:
\begin{itemize}
\item {any finite-dimensional representation of $\sl _2$ is contained in the local part of the graded Lie algebra $L\left (\sl _2, \mathrm {f.d.}\right )=\bigoplus _{n\in \Z }L_n$, defined in (\ref {sl2fd_1})},
\item {the structure of $L\left (\sl _2, \mathrm {f.d.}\right )$ is expressed by the matrices $\widetilde {A},\widetilde {D}, \widetilde {\Gamma }$ and $\widetilde {C}=\widetilde {\Gamma }\cdot {}^t \widetilde {D}\cdot \widetilde{A}\cdot \widetilde {D}$},
\item {the Lie algebra $L\left (\sl _2, \mathrm {f.d.}\right )$ is ``closed'' in the sense of Theorem \ref {th;sl2fd_closed}}.
\end{itemize}

%
%
%
%
%

\end{document}